\documentclass[11pt,english]{smfart}
\usepackage[T1]{fontenc}
\usepackage[english,francais]{babel}
\usepackage{url}
\usepackage{xcolor}
\usepackage[all,cmtip]{xy}
\usepackage{graphicx}

\usepackage{amssymb,url,xspace,smfthm}
\usepackage{anysize}
\usepackage{sansmathaccent}

\makeatother

\newcommand{\BibTeX}{{\scshape Bib}\kern-.08em\TeX}
\newcommand{\T}{\S\kern .15em\relax }
\newcommand{\AMS}{$\mathcal{A}$\kern-.1667em\lower.5ex\hbox
        {$\mathcal{M}$}\kern-.125em$\mathcal{S}$}

\newcommand{\NE}{\operatorname{NE}}

\newcommand{\Exc}{\operatorname{Exc}}

\title{A characterization of some Fano 4-folds through conic fibrations}


\author{Pedro \textsc{Montero}}
\address{Departamento de Matem\'atica, Universidad T\'ecnica Federico Santa Mar\'ia, Valpara\'iso, Chile.}
\email{pedro.montero@usm.cl}

\author{Eleonora Anna \textsc{Romano}}
\address{Institute of Mathematics, University of Warsaw, Warszawa, Poland.}
\email{elrom@mimuw.edu.pl}

\begin{document}
\def\smfbyname{}

\maketitle

\begin{abstract}
We find a characterization for Fano 4-folds $X$ with Lefschetz defect $\delta_{X}=3$: besides the product of two del Pezzo surfaces, they correspond to varieties admitting a conic bundle structure $f\colon X\to Y$ with $\rho_{X}-\rho_{Y}=3$. Moreover, we observe that all of these varieties are rational. We give the list of all possible targets of such contractions. Combining our results with the classification of toric Fano $4$-folds due to Batyrev and Sato we provide explicit examples of Fano conic bu
%
%
ndles from toric $4$-folds with $\delta_{X}=3$. 
\end{abstract}

\tableofcontents

\pagebreak
\section{Introduction}

Let $X$ be a complex smooth Fano variety of dimension $n$. 
Let us denote by $\mathcal{N}_{1}(X)$ the $\mathbb{R}$-vector space of one-cycles with real coefficients, modulo numerical equivalence, whose dimension is the \textit{Picard number} $\rho_{X}$. 

Let $D\subset X$ be a prime divisor. The inclusion $i\colon D\hookrightarrow X$ induces a pushforward of one-cycles $i_{*}\colon \mathcal{N}_{1}(D)\to \mathcal{N}_{1}(X)$. We set $\mathcal{N}_{1}(D,X):=i_{*}(\mathcal{N}_{1}(D))\subseteq \mathcal{N}_{1}(X)$, which is the linear subspace of $\mathcal{N}_{1}(X)$ spanned by numerical classes of curves contained in $D$. In \cite{CAS1} Casagrande introduced the following invariant, called \textit{Lefschetz defect}: 
\begin{center}
$\delta_{X}:=\max\{\operatorname{codim} \mathcal{N}_{1}(D,X)\;|\; D\subset X \text{ prime divisor} \}$. 
\end{center}

This invariant allows to deduce many features about the variety. Indeed, in \cite[Theorem 1.1]{CAS1} Casagrande proved that if $\delta_{X}\geq 4$ then $X\cong S\times T$, where $S$ is a del Pezzo surface and $T$ is a $(n-2)$-dimensional smooth Fano variety. We refer the reader to \cite{CAS1} and \cite{NOCE} for the properties of $\delta_{X}$. In \cite{IO} Romano studied non-elementary Fano conic bundles, namely fiber type contractions $f\colon X\to Y$ with $X$ as above, such that the fibers of $f$ are one-dimensional, and $\rho_{X}-\rho_{Y}>1$. We recall that in this case $Y$ is a smooth variety by \cite[Theorem 3.1]{AND85} (see also \cite[Theorem 3.1.3]{TESI}).

An important point for our purposes is that if $X$ admits a conic bundle structure $f\colon X\to Y$ with relative Picard dimension $r:=\rho_{X}-\rho_{Y}\geq 3$, there is a relation between $r$ and $\delta_{X}$. More precisely, it is possible to find some lower-bounds for $\delta_{X}$ in terms of the relative Picard dimension of $f$. For instance, under this assumption we get $\delta_{X}\geq r-1$ by \cite[Lemma 3.10]{IO}. 

In this paper we focus on the case in which $X$ is a Fano $4$-fold, namely a $4$-dimensional complex smooth Fano variety. Our main goal is to find a characterization in terms of the Lefschetz defect of some Fano $4$-folds admitting a conic fibration. 

The first case of main interest is when $\delta_{X}=3$. Indeed if $\delta_{X}\geq 4$, we have already observed that $X\cong S_{1}\times S_{2}$ where each $S_{i}$ is a del Pezzo surface. In this situation all conic bundle structures are well known, and it is easy to find bounds for $\delta_{X}$. We refer the reader to Proposition \ref{prop2} and Remark \ref{case_product} for details. 

When $X$ is a Fano manifold with $\delta_{X}=3$, in the proof of \cite[Theorem 1.1 (ii)]{CAS1} Casagrande showed that there exists a conic bundle structure $f\colon X\to Y$ where $\rho_{X}-\rho_{Y}=3$. In this paper we prove that the converse holds for Fano $4$-folds such that $X\ncong S_{1}\times S_{2}$, where each $S_{i}$ is a del Pezzo surface. Combining what is already known with our results, we formulate the main Theorem of this paper as follows.

\begin{theo} \label{main_thm} Let $X$ be a Fano 4-fold such that $X\ncong S_{1}\times S_{2}$, where each $S_{i}$ is a del Pezzo surface. Then $\delta_{X}=3$ if and only if there exists a conic bundle $f\colon X\to Y$ such that $\rho_{X}-\rho_{Y}=3$. 
\end{theo}

As we have already stressed, in order to prove Theorem \ref{main_thm} we are left to investigate in more details conic bundles of Fano $4$-folds, with relative Picard dimension $3$. By \cite[Corollary of Proposition 4.3]{WIS} the target of such contraction is a Fano $3$-fold. One of the main ingredient is the following result where we give the list of all possible targets of such contractions.

\begin{prop} \label{prop1} Let $f\colon X\to Y$ be a Fano conic bundle, where $\dim{(X)}=4$ and $\rho_{X}-\rho_{Y}=3$. Then $5\leq\rho_{X}\leq 13$. Moreover:
\begin{itemize}	
\item [(a)] If $\rho_{X}=5$, then $Y$ is one of the following Fano $3$-folds: $Y\cong \mathbb{P}^{1}\times \mathbb{P}^{2}$; {$Y\cong \mathbb{P}_{\mathbb{P}^{2}}(\mathcal{O}\oplus \mathcal{O}(1))$}; $Y\cong \mathbb{P}_{\mathbb{P}^{2}}(\mathcal{O}\oplus \mathcal{O}(2))$. 

\item [(b)] If $\rho_{X}=6$, then $Y$ is one of the following Fano $3$-folds: $Y\cong \mathbb{P}^{1}\times\mathbb{P}^{1}\times \mathbb{P}^{1}$; $Y\cong \mathbb{F}_{1}\times \mathbb{P}^{1}$; $Y\cong \mathbb{P}_{\mathbb{P}^{1}\times \mathbb{P}^{1}}(\mathcal{O}(-1,-1)\oplus \mathcal{O})$; $Y\cong \mathbb{P}_{\mathbb{P}^{1}\times \mathbb{P}^{1}}(\mathcal{O}(0,-1)\oplus \mathcal{O}(-1,0))$. 

\item [(c)] If $\rho_{X}\geq 7$, then $X\cong S_{1}\times S$, where $S_{1}$ is a del Pezzo surface with $\rho_{S_{1}}=4$, $Y\cong \mathbb{P}^{1}\times S$, and $f$ is induced by a conic bundle $S_{1}\to \mathbb{P}^{1}$.
\end{itemize}
\end{prop}

From the previous results, we are able to deduce the following:

\begin{coro} \label{cor_1} Let $X$ be a Fano $4$-fold with $\delta_{X}=3$ or such that there exists a conic bundle $X\to Y$ with $\rho_{X}-\rho_{Y}=3$. Then $X$ is rational. 
\end{coro}

Using the definition of $\delta_{X}$, and Theorem \ref{main_thm} we give a characterization of Fano $4$-folds admitting a conic bundle with relative Picard number $3$, in terms of the Picard number of prime divisors on $X$.

\begin{theo} \label{thm5} Let $X$ be a Fano 4-fold, such that $X\ncong S_{1}\times S_{2}$, where each $S_{i}$ is a del Pezzo surface. Then $X$ admits a conic bundle structure $f\colon X\to Y$ with $\rho_{X}-\rho_{Y}=3$ if and only if there exists a prime divisor $D$ on $X$ such that $\rho_{D}=\rho_{X}-3$. If one the above equivalent conditions holds, then $X$ is rational.  
\end{theo}

On one hand, Theorem \ref{main_thm} gives a characterization of Fano $4$-folds with $\delta_{X}=3$, and allows to deduce the rationality result stated in Corollary \ref{cor_1}. On the other hand, we use Theorem \ref{main_thm}, Proposition \ref{prop1} and some properties of Fano conic bundle to get results in the perspective of classification of Fano 4-folds with $\delta_{X}=3$. The following proposition is a first step in this direction, in the case in which $X$ is a Fano $4$-fold with $\rho_{X}=5$ admitting a conic bundle structure of relative Picard dimension $3$. 

\begin{theo} \label{thm2} Let $f\colon X\to Y$ be a Fano conic bundle where $\dim{(X)}=4$, $\rho_{X}=5$ and $\rho_{X}-\rho_{Y}=3$. Then either $X$ is obtained as the blow-up along two smooth disjoint surfaces isomorphic to $\mathbb{P}^2$ contained in one of the following 4-folds:
\begin{itemize}
\item [(1)] $\mathbb{P}^{1}\times Z$, where $Z$ is a Fano $\mathbb{P}^{1}$-bundle over $\mathbb{P}^{2}$;
\item [(2)] $\mathbb{P}_{Y}(\mathcal{E})$, where $\mathcal{E}$ is a rank $2$ vector bundle on $Y$, where either $Y\cong \mathbb{P}_{\mathbb{P}^{2}}(\mathcal{O}\oplus \mathcal{O}(1))$ or $Y\cong \mathbb{P}_{\mathbb{P}^{2}}(\mathcal{O}\oplus \mathcal{O}(2))$;
\end{itemize}
or $X$ is obtained as the blow-up of three disjoint sections of the $\mathbb{P}^{2}$-bundle $\mathbb{P}_{\mathbb{P}^{2}}(\mathcal{O}\oplus \mathcal{O}(a)\oplus \mathcal{O}(b))$, where $(a,b)\in \{(0,0), (0,1), (0,2), (1,1),(1,2)\}$. 
\end{theo}

\subsection*{Methods and outline of the article} Let us describe in more detail the content of this paper. In Section 2 we set up notation and terminology. In Subsection \ref{defect} and \ref{conic_section} we recall some crucial facts about the Lefschetz defect and Fano conic bundles. 

Section 3 contains the central part of the paper: we prove Theorem \ref{main_thm}, Proposition \ref{prop1}, and we deduce Corollary \ref{cor_1}. We conclude this section by proving Theorem \ref{thm5}. 

Let us summarize the strategy to show Theorem \ref{main_thm}. If $\delta_{X}=3$, the thesis follows by the proof of \cite[Theorem 1.1 (ii)]{CAS1}, as we observe in Proposition \ref{first_implication}. 

Hence the main point to reach the statement consists in the study of non-elementary Fano conic bundles $f\colon X\to Y$ with $\rho_{X}-\rho_{Y}=3$. To this end, we use Proposition \ref{prop1}. The key to prove Proposition \ref{prop1} is to look at the \textit{discriminant divisor} $\bigtriangleup_{f}$ of $f$, which is a divisor of $Y$ that was introduced in \cite{B} and studied further in \cite[\S 1]{SARK}. See also \cite[$\S$3.2]{TESI} for a complete exposition about the main geometric properties related to $\bigtriangleup_{f}$. The discriminant divisor of $f$ is defined as follows:
\begin{center}
	$\bigtriangleup_{f}=\{y\in Y\mid f^{-1}(y) \text{ is singular}\}$.
\end{center}  

More precisely, we look at the components of $\bigtriangleup_{f}$ that in our case are two smooth disjoint divisors of $Y$, as we recall in Subsection \ref{conic_section}. Let us denote them by $A_{i}$ for $i=1,2$. Since $Y$ is Fano, by \cite[Corollary 1.3.2]{BIRKAR} we know that it is a Mori Dream Space (see \cite{HU} for more details about Mori Dream Spaces). 

Applying some techniques of Mori Dream Spaces to these divisors $A_{i}$, we deduce what kind of contractions $Y$ could have. 

Then by considering some results on conic bundles we are able to deduce that $Y$ is a Fano bundle of rank $2$ over $S$, that means that $Y\cong \mathbb{P}_{S}(\mathcal{F})$, where $\mathcal{F}$ is a rank $2$ vector bundle over $S$. At this point, using the main Theorem of \cite{S_WIS} which gives the list of Fano bundles of rank $2$ on surfaces we get the proof of Proposition \ref{prop1}. Moreover in many cases we are able to deduce an explicit description of the discriminant divisor of $f$ (see Corollary \ref{corollary_prop1}), and to compute $\rho_{A_i}$. Then this information, together with the conic bundle structure of $X$ allow to deduce that $\delta_{X}=3$, and hence Theorem \ref{main_thm}. 

Section 4 is entirely devoted to the case in which $X$ is a Fano $4$-fold with $\rho_{X}=5$ admitting a conic bundle structure of relative Picard dimension $3$. Here we prove some results (see Proposition \ref{lemma2} and Lemma \ref{lemma3}) which we need to show Theorem \ref{thm2}.

The aim of Section 5 is to provide examples of Fano conic bundles $f\colon X\to Y$ where $X$ is a toric Fano $4$-fold with $\rho_{X} \in \{5,6\}$, and $\rho_{X}-\rho_{Y}=3$. In particular we deduce that every $Y$ in Proposition \ref{prop1} appears as target of some conic bundle of relative Picard dimension $3$. 

\section{Preliminaries}

\subsection{Notations and Conventions}

We work over the field of complex numbers. Let $X$ be a manifold, \emph{i.e.} a smooth projective variety, of arbitrary dimension $n$. We call $X$ a \textit{Fano variety} if $-K_{X}$ is an ample divisor.

$\mathcal{N}_{1}(X)$ (respectively, $\mathcal{N}^{1}(X)$) is the $\mathbb{R}$-vector space of one-cycles (respectively, divisors) with real coefficients, modulo numerical equivalence.

$\dim \mathcal{N}_{1}(X)=\dim \mathcal{N}^{1}(X)=\rho_{X}$ is the Picard number of $X$.

Let $C$ be a one-cycle of X, and $D$ a divisor of X. We denote by $[C]$ (respectively, $[D]$) the numerical equivalence class in $\mathcal{N}_{1}(X)$ (respectively, $\mathcal{N}^{1}(X)$). Moreover we denote by $\mathbb{R}[C]$ the linear span of $[C]$ in $\mathcal{N}_{1}(X)$, and by $\mathbb{R}_{\geq0}[C]$ the corresponding ray. The symbol $\equiv$ stands for numerical equivalence (for both one-cycles and divisors).

$\operatorname{NE}(X)\subset \mathcal{N}_{1}(X)$ is the convex cone generated by classes of effective curves. 

We denote by $\operatorname{Eff} (X)$ the \textit{effective cone} of $X$, that is the convex cone inside $\mathcal{N}^{1}(X)$ spanned by classes of effective divisors.

Denote by $\operatorname{Bs}(|\cdot|)$ the base locus of a linear system. Let $D\subset X$ be a divisor. The divisor $D$ is movable if there exists $m > 0$, $m\in \mathbb{Z}$ such that $\operatorname{codim} \operatorname{Bs} |mD| \geq 2$.
We denote by $\operatorname{Mov}(X)$ the \textit{movable cone} of $X$, that is the closure of the convex cone of $\mathcal{N}^{1}(X)$ spanned by classes of movable divisors. We denote by $\operatorname{Nef}(X)$ the \textit{nef cone} of $X$, that is the convex cone of $\mathcal{N}^{1}(X)$ spanned by classes of nef divisors.

A \textit{contraction} of $X$ is a surjective morphism $\varphi\colon X\to Y$ with connected fibers, where $Y$ is normal and projective.
We denote by $\text{Exc}(\varphi)$ the exceptional locus of $\varphi$, i.e. the locus where $\varphi$ is not an isomorphism.
We say that $\varphi$ is of type $(a, b)$ if $\dim\text{Exc}(\varphi)=a$ and $\dim \varphi(\text{Exc}(\varphi))=b$.

We denote by $\text{NE}(\varphi)$ the relative cone of $f$, namely the convex subcone of $\text{NE}(X)$ generated by classes of curves contracted by $\varphi$.

A contraction of $X$ is called \textit{$K_{X}$-negative} (or simply \textit{$K$-negative}) if $-K_{X}\cdot C>0$ for every curve $C$ contracted by $\varphi$.

A $\mathbb{P}^{1}$-bundle over a projective variety $Z$ is the projectivization of a rank $2$ vector bundle on $Z$. 
We say that a vector bundle $\mathcal{E}$ of rank $r$ on a projective variety $Z$ is a \textit{Fano bundle} of rank $r$ if $\mathbb{P}_{Z}(\mathcal{E})$ is a Fano manifold. By abuse of notation, we will often say in this case that the variety $\mathbb{P}_{Z}(\mathcal{E})$ is a Fano bundle as well. 

\subsection{Lefschetz defect}\label{defect}

Let $X$ be a smooth Fano variety and take $D\subset X$ a prime divisor. The inclusion $i\colon D\hookrightarrow X$ induces a pushforward of one-cycles $i_{*}\colon \mathcal{N}_{1}(D)\to \mathcal{N}_{1}(X)$, that does not need to be injective nor surjective. 

We set $\mathcal{N}_{1}(D,X):=i_{*}(\mathcal{N}_{1}(D))\subseteq \mathcal{N}_{1}(X)$. Equivalently, $\mathcal{N}_{1}(D,X)$ is the linear subspace of $\mathcal{N}_{1}(X)$ spanned by classes of curves contained in $D$. 

Working with $\mathcal{N}_{1}(D,X)$ instead  
$\mathcal{N}_{1}(D)$ means that we consider curves in $D$ modulo numerical equivalence in $X$, instead of numerical equivalence in $D$. Note that $\dim{\mathcal{N}_{1}(D,X)}\leq \rho_{D}$. 

In \cite{CAS1} Casagrande introduced the following invariant of $X$, called \textit{Lefschetz defect}:
\begin{center}
	$\delta_{X}:=\max{\{\operatorname{codim}{\mathcal{N}_{1}(D,X)\mid D \text{ is a prime divisor of X}}\}}$
\end{center}

Note that if $n=\dim(X)\geq 3$ and $D$ is ample, then $i_{*}\colon \mathcal{N}_{1}(D)\to \mathcal{N}_{1}(X)$ is surjective by the Lefschetz Theorem on hyperplane sections. Hence $\delta_{X}$ measures the failure of this Theorem for non ample divisors.

The following theorem is related to this invariant and will be crucial for our purposes.

\begin{theo}[\cite{CAS1}, Theorem 1.1] \label{prodotto}
	For any Fano manifold $X$ we have $\delta_{X}\leq 8$. Moreover: 
    \begin{itemize}
    \item [(a)] If $\delta_{X}\geq 4$ then $X\cong S\times T$, where $S$ is a del Pezzo surface, $\rho_{S}=\delta_{X}+1$, and $\delta_{T}\leq \delta_{X}$;
    \item [(b)] If $\delta_{X}= 3$ then there exists a flat fiber type contraction $\psi\colon X\to Z$ where $Z$ is an $(n-2)$-dimensional Fano manifold, $\rho_{X}-\rho_{Z}=4$, and $\delta_{Z}\leq 3$.
    \end{itemize}
\end{theo}

\subsection{Fano conic bundles} \label{conic_section}
In this subsection we recall some definitions and the main properties on Fano conic bundles which we need throughout the paper. We refer the reader to \cite{IO, IO2, TESI} for a more complete exposition about Fano conic bundles. For the convenience of the reader we recall some results from \cite{IO} without proofs, thus making our exposition self-contained. 

\begin{defi} \label{def_cb}
	Let $X$ be a smooth, projective variety and let $Y$ be a normal, projective variety. A \textit{conic bundle} $f\colon X\to Y$ is a fiber type $K$-negative contraction where every fiber is one-dimensional. A \textit{Fano conic bundle} is a conic bundle $f\colon X\to Y$ where $X$ is a Fano variety.
\end{defi}

\begin{rema}
By \cite[Theorem 3.1]{AND85} it follows that for a conic bundle $f\colon X\to Y$ as in Definition \ref{def_cb} the variety $Y$ is actually smooth. See also \cite[Theorem 3.1.3]{TESI} for a detailed proof concerning the smoothness of $Y$. 
\end{rema}

\begin{defi} 
Let $f\colon X\to Y$ be a conic bundle. If $\rho_{X}-\rho_{Y}> 1$, $f$ is called non-elementary. Otherwise $f$ is called elementary.
\end{defi} 

\begin{prop} [\cite{IO}, Proposition 3.5]\label{factorization} 
Let $f\colon X\to Y$ be a Fano conic bundle with $\rho_{X}-\rho_{Y}:=r$. Then:
\begin{itemize}
\item [(a)] There exists a factorization: $X\stackrel{f_{1}}{\to} X_{1} \to \dots \to X_{r-2}\stackrel{f_{r-1}}{\to} X_{r-1}\stackrel{g}{\to} Y$, where each $f_{i}$ is a blow-up of the smooth variety $X_{i+1}$ along a smooth subvariety of codimension $2$, and $g$ is an elementary conic bundle;
\item [(b)] There exist smooth prime divisors $A_{1}, \dots A_{r-1}$ of $Y$ such that $f^{*}(A_{i})=E_{i}+ \hat{E}_{i}$, where $E_{i}$, $\hat{E}_{i}$ are smooth prime divisors on $X$ such that $E_{i}\to A_{i}$ and $\hat{E}_{i}\to A_{i}$ are $\mathbb{P}^{1}$-bundles for every $i=1,\dots,r$;
\item [(c)] Set $\bigtriangleup_{f}:=\{y\in Y| f^{-1}(y) \ \text{is singular}\}$ the discriminant divisor of $f$.
Then $\bigtriangleup_{f}=A_{1}\sqcup A_{2}\sqcup \dots \sqcup A_{r-1} \sqcup \bigtriangleup_{g}$, where $A_{i}$ are smooth components of $\bigtriangleup_{f}$ for $i=1,\dots,r-1$, and $\bigtriangleup_{g}$ is the discriminant divisor of $g$. 
\end{itemize}
\end{prop}

The geometric situation is represented in Figure \ref{fig:fig1}.
\begin{figure}
	\centering
    \includegraphics[scale=0.8]{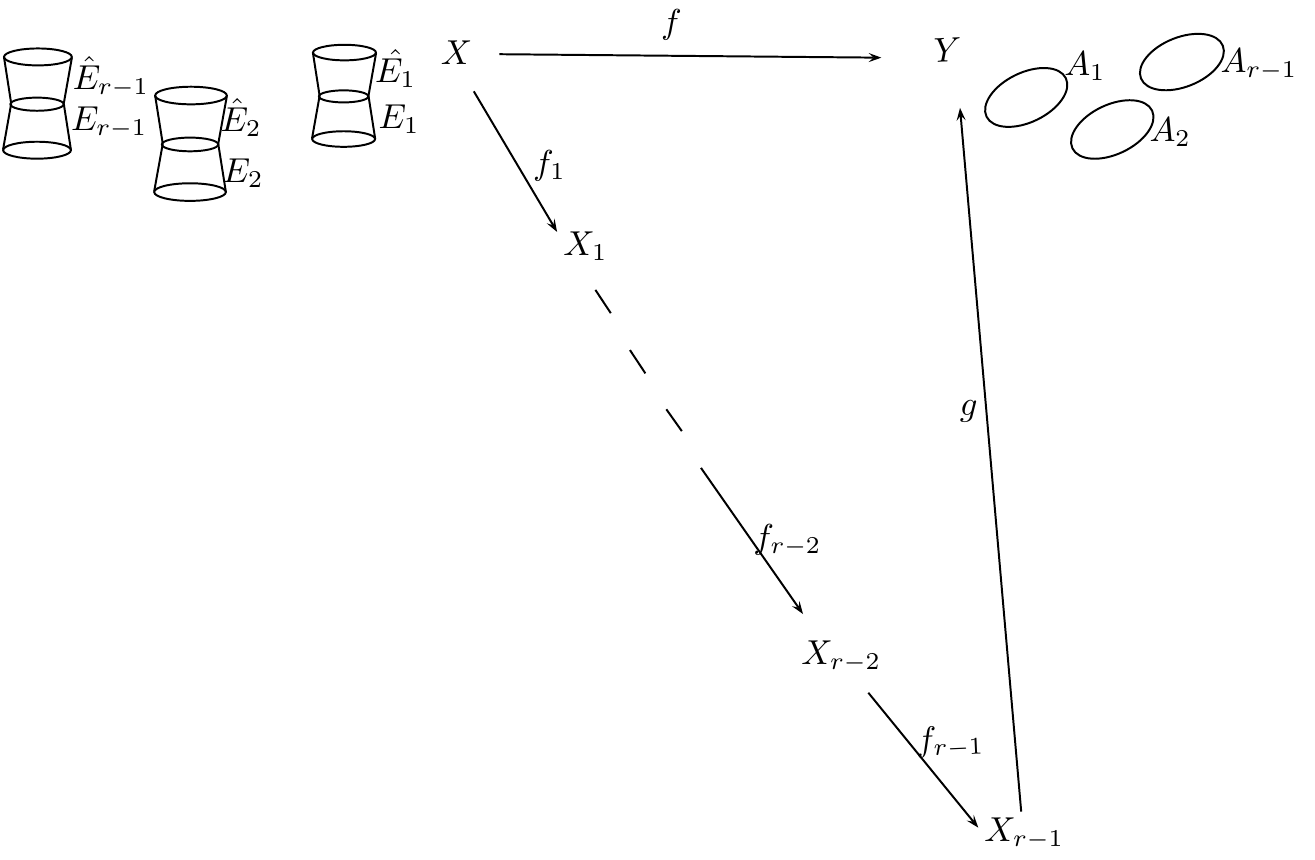}
    \caption{Divisors of Proposition \ref{factorization}.}
	\label{fig:fig1}
\end{figure}

Finally, we recall a weaker version of \cite[Theorem 4.2]{IO}, which we need in the next sections.  

\begin{theo} \label{thm_tesi} Let $f\colon X\to Y$ be a Fano conic bundle. Then $\rho_{X}-\rho_{Y}\leq 8$. Moreover:
\begin{itemize} 
\item [(a)] If $\rho_{X}-\rho_{Y}\geq 4$, then $X\cong S\times T$, where $S$ is a del Pezzo surface, $T$ is a $(n-2)$-dimensional Fano manifold, $Y\cong \mathbb{P}^{1}\times T$, and $f$ is induced by a conic bundle $S\to \mathbb{P}^{1}$;
\item [(b)] Assume that $\rho_{X}-\rho_{Y}=3$. Let us take a factorization for $f$ as in Proposition \ref{factorization} (a), and let us denote by $g\colon X_{2}\to Y$ the elementary conic bundle in this factorization. Then $g$ is smooth, $Y$ is also Fano and there exists a smooth $\mathbb{P}^{1}$-fibration\footnote{A smooth $\mathbb{P}^{1}$-fibration is a smooth morphism such that every fiber is isomorphic to $\mathbb{P}^{1}$.} $\xi\colon Y\to {Y}^{\prime}$, where ${Y}^{\prime}$ is a $(n-2)$-dimensional Fano manifold. 
\end{itemize}
\end{theo}

The following Proposition is a consequence of the proof of \cite[Theorem 1.1 (ii)]{CAS1}.  

\begin{prop} \label{first_implication} Let $X$ be a Fano manifold of dimension $n\geq 3$. If $\delta_{X}=3$, then it admits a conic bundle structure $f\colon X\to Y$ where $\rho_{X}-\rho_{Y}=3$.
\end{prop}

\begin{proof} By Theorem \ref{prodotto} $(b)$ there exists a flat fiber type contraction $\psi\colon X\to Z$ where $Z$ is a $(n-2)$-dimensional Fano manifold and $\rho_{X}-\rho_{Z}=4$. By the proof of \cite[Theorem 1.1 (ii)]{CAS1}\footnote{More precisely, see the proof of \cite[Proposition 3.3.1]{CAS1}, whose outline is given in \cite[\S 3.3.3]{CAS1}.} it follows that this fiber type contraction $\psi$ factors in the following way $\psi=h\circ f$ where $f\colon X\to Y$ and $h\colon Y\to Z$ are contractions. In particular, $f$ is a Fano conic bundle with $\rho_{X}-\rho_{Y}=3$, and our claim follows.
\end{proof}

\section{Main results} \label{main_section}

In this section we discuss the main results of this paper which give a characterization of some Fano 4-folds, in terms of the Lefschetz defect of these varieties. We first analyze the easier case in which the Fano variety $X$ is a product between two del Pezzo surfaces. The first proposition is an immediate consequence of the structure of conic bundles induced on del Pezzo surfaces, and of the very definition of $\delta_{X}$. We point out that in this case we know all possible non-elementary conic bundles of $X$. 

\begin{prop} \label{prop2} Let $X\cong S_{1}\times S_{2}$ where $S_{i}$ are smooth del Pezzo surfaces $i=1,2$, and $X\ncong \mathbb{P}^{2}\times \mathbb{P}^{2}$. Then $X$ has a conic bundle structure $f\colon X\to Y$ with $\delta_{X}\geq r:=\rho_{X}-\rho_{Y}$.
\end{prop}

\begin{proof} Let $X\cong S_{1}\times S_{2}$ where $S_{i}$ are del Pezzo surfaces $i=1,2$. Take one between $S_{i}$ such that $S_{i}\ncong \mathbb{P}^{2}$. Assume for simplicity that it is $S_{1}$. Then we can consider the following contraction $f=(\widetilde{f}, \operatorname{id}_{S_{2}})\colon X\to \mathbb{P}^{1}\times S_{2}$, where $\widetilde{f}\colon S_{1}\to \mathbb{P}^{1}$ is a conic bundle with $r=\rho_{S_{1}}-1$, and $\operatorname{id}_{S_{2}}\colon S_{2}\to S_{2}$ is the identity. The statement follows because $\delta_{X}=\max\{\rho_{S_{1}}-1,\rho_{S_{2}}-1\}$ (see \cite[Example 3.1]{CAS1}).
\end{proof}

\begin{rema} \label{case_product} Let $X$ be as in the above Proposition. We know all possible conic bundles of $X$. Indeed, if $f\colon X\to Y$ is a Fano conic bundle then by \cite[Lemma 2.10]{IO} $Y$ is also a product of two smooth varieties. In particular, the same proof of \cite[Theorem 4.2 $(1)$]{IO} allows to deduce that $Y\cong \mathbb{P}^{1}\times S_{2}$, and $f$ is induced by a conic bundle $S_{1}\to \mathbb{P}^{1}$. 
\end{rema}

Our next goal will be to analyze the case in which $X$ is not isomorphic to a product of two del Pezzo surfaces, by proving Theorem \ref{main_thm}. To this end, we need two preliminary results given by the following Lemma and Proposition \ref{prop1}.  

\begin{lemm} \label{lemma1} Let $Y=\mathbb{P}_{S}(\mathcal{E})$ be a Fano bundle of rank $2$ on a smooth del Pezzo surface $S$. Assume that $Y$ has an elementary divisorial contraction $\psi\colon Y\to \mathbb{P}^{1}\times \mathbb{P}^{2}$ of type $(2,1)$, and that there exists a prime effective divisor $D$ of $Y$ such that $D\cap \operatorname{Exc}(\psi)=\emptyset$. Then $S\cong \mathbb{F}_{1}$ and $Y\cong \mathbb{P}^{1}\times \mathbb{F}_{1}$.
\end{lemm}

\begin{proof} Let us denote by $A:=\text{Exc}(\psi)$. We set $C:=\psi(A)$ and $B:=\psi(D)$ which by our assumption are respectively a curve and a divisor in $\mathbb{P}^{1}\times \mathbb{P}^{2}$. Note that $C\cap B=\emptyset$ because $A\cap D=\emptyset$. Since $C$ is a curve on $\mathbb{P}^{1}\times \mathbb{P}^{2}$ disjoint from an effective non-zero divisor, $C$ is either of the form $\mathbb{P}^1\times \{p\}$ or $\{q\}\times C' $, where $p\in \mathbb{P}^2$ and $q\in \mathbb{P}^1$ are points and $C'\subset \mathbb{P}^2$ is a curve. We prove that the second case is not possible. Assume by contradiction that it holds. Let us denote by $G$ the transform of $\{q'\}\times \mathbb{P}^2$ on $Y$, where $q'$ is a point in $\mathbb{P}^1$ different by $q$. Then $G\cap A=\emptyset$ and thus $G\cong \mathbb{P}^2$ does not dominate $S$ via the $\mathbb{P}^1$-bundle $Y\to S$ since $G$ cannot be a $\mathbb{P}^1$-bundle over a curve. But this is impossible because $\rho_S=\rho_Y-1=2$. We conclude that $C\cong \mathbb{P}^1\times \{p\}$ where $p\in \mathbb{P}^2$ is a point, and our thesis follows. 
\end{proof}

Now we prove Proposition \ref{prop1}. The strategy is to look at the two components of $\bigtriangleup_{f}$, which are smooth divisors of $Y$, and that we denote by $A_{i}$ as in Proposition \ref{factorization} $(c)$. Since by \cite[Corollary of Proposition 4.3]{WIS} $Y$ is Fano, using \cite[Corollary 1.3.2]{BIRKAR} we know that it is a Mori Dream Space (see \cite{HU} for details about Mori Dream Spaces). Applying some techniques of Mori Dream Spaces to these divisors $A_{i}$, we deduce what kind of contractions $Y$ could have. Moreover, using the smooth $\mathbb{P}^{1}$-fibration $\xi\colon Y\to S$ of Theorem \ref{thm_tesi} $(b)$ with $S$ a del Pezzo surface, we are able to deduce that $Y$ is a Fano bundle of rank $2$ over $S$. We get the statement, putting together all geometric information which arise from the Mori Dream Spaces tools, and using the main Theorem of \cite{S_WIS}, which gives the list of Fano bundles of rank $2$ on surfaces.

\begin{proof}[\textbf{\bf{Proof of Proposition \ref{prop1}}}] By Theorem \ref{thm_tesi} $(b)$, there exists a smooth $\mathbb{P}^{1}$-fibration $\xi\colon Y\to S$, with $S$ a del Pezzo surface. By our assumption $\rho_{X}-\rho_{S}=4$ so that $\rho_{X}\geq 5$. Since $\xi$ is smooth with fibers isomorphic to $\mathbb{P}^{1}$ and $S$ is rational, then there exists a rank $2$ vector bundle $\mathcal{E}$ on $S$ such that $Y\cong \mathbb{P}_{S}(\mathcal{E})$ (see for instance \cite[Proposition 4.3]{FUJ16}).

Assume that $\rho_{X}=5$. This means that in the smooth $\mathbb{P}^{1}$-fibration $\xi\colon Y\to S$, $S \cong \mathbb{P}^{2}$. As we have already observed, $Y\cong \mathbb{P}_{\mathbb{P}^{2}}(\mathcal{E})$. We find all possible $\mathcal{E}$. Let us denote by $A_{i}$ for $i=1,2$ the two smooth components of the discriminant divisor of $f$, as in Proposition \ref{factorization} $(c)$. Suppose that both divisors $A_{i}$ are nef. We prove that in this case $Y\cong \mathbb{P}^{1}\times \mathbb{P}^{2}$. 

Since by \cite[Corollary of Proposition 4.3]{WIS} $Y$ is Fano, applying \cite[Corollary 1.3.2]{BIRKAR} it is a MDS, hence $A_{i}$ are semiample divisors. Using that $A_{1}\cap A_{2}=\emptyset$, it is easy to check that they are linearly proportional divisors of $Y$ which give a contraction $\Psi\colon Y\to \mathbb{P}^{1}$ such that $\Psi(A_{i})$ is a point for $i=1,2$. We observe that $\xi(A_{i})=S$ for $i=1,2$. Indeed, by Step 2 of the proof of \cite[Theorem 4.2 (2)]{IO} there exists a prime divisor $G_1\subset X$ which is a $\mathbb{P}^{1}$-bundle which dominates $Y$ through $f\colon X\to Y$, with fiber $g_1\subset G_1$ such that $f^{*}{(A_i)}\cdot g_1>0$, and by Step 3 of the same proof it follows that $\NE{(\xi)}=\mathbb{R}_{\geq 0}[f(g_1)]$ (see also Fig.2 in \cite{IO}). Then $\Psi$ is finite on the fiber of $\xi$. The claim follows by \cite[Lemma 4.9]{CAS4}. 

Assume that only one divisor between $A_{i}$'s is nef. For instance, suppose that it is $A_{1}$. Since $Y$ is a Fano 3-fold it has no small contractions, and the description of $\text{Nef}(Y)$ with $Y$ a MDS (see \cite[Def. 3.12 (3)]{MDS_NOTES}) together with \cite[Corollary 3.12]{MDS_NOTES} gives $\text{Nef}(Y)=\text{Mov}(Y)$. Therefore $A_{2}$ is not a movable divisor and by \cite[Remark 4.7]{MDS_NOTES} we have that $A_{2}$ is an exceptional divisor of a birational contraction of $Y$. 

More precisely, $A_{1}$ gives this birational contraction which contracts the other divisor $A_{2}$ to a point. Indeed, since $A_1$ is semiample, the linear system $|mA_1|$ defines a contraction $\varphi:= \varphi_{|mA_1|}\colon Y\to T$ for $m\gg 0$, $m\in \mathbb{N}$ with $mA_1=\varphi^* (A_T)$ for some ample divisor $A_T$ on $T$. Being $A_1\cap A_2 = \emptyset$, for every curve $C\subset A_2$ one has $mA_1\cdot C=0$, and using the projection formula we get $\varphi^* (A_T)\cdot C= A_T\cdot \varphi_*(C)=0$ then $A_2$ is contracted to a point in $T$. By \cite[Lemma 3.9]{CD15} and the main Theorem of \cite{S_WIS} it follows that either $Y\cong \mathbb{P}_{\mathbb{P}^{2}}(\mathcal{O}\oplus \mathcal{O}(1))$ or $Y\cong \mathbb{P}_{\mathbb{P}^{2}}(\mathcal{O}\oplus \mathcal{O}(2))$. 

We observe that there are no other possibilities: if neither $A_{1}$ or $A_{2}$ are nef divisors of $Y$, then they should be two exceptional divisors of two different extremal divisorial contractions of $Y$, but $Y$ is a Fano bundle of rank $2$ over $\mathbb{P}^{2}$ and we get a contradiction. Then claim $(a)$ follows. 

Assume that $\rho_{X}=6$. Applying again \cite[Corollary of Proposition 4.3]{WIS}, $Y$ is Fano and by Theorem \ref{main_thm} (b) there exists a smooth $\mathbb{P}^{1}$-fibration $\xi\colon Y\to S$, where either $S\cong \mathbb{P}^{1}\times \mathbb{P}^{1}$ or $S\cong \mathbb{F}_{1}$. As above, we denote by $A_{i}$ the two smooth components of the discriminant divisor of $f$, as in Proposition \ref{factorization} $(c)$. 

Suppose that both divisors $A_{i}$ are nef. The same proof of point $(a)$ allows us to deduce that $Y\cong S\times \mathbb{P}^{1}$ where $S$ is isomorphic to $\mathbb{P}^{1}\times \mathbb{P}^{1}$ or to $\mathbb{F}_{1}$. 
Assume that only one divisor between the $A_{i}$'s is nef. For simplicity, suppose that it is $A_{1}$. As above, we deduce that $A_{1}$ gives a birational contraction which contracts $A_{2}$ to a point. By \cite[Lemma 3.9]{CD15} and the main Theorem of \cite{S_WIS} it follows that $Y\cong \mathbb{P}_{\mathbb{P}^{1}\times \mathbb{P}^{1}}(\mathcal{O}(-1,-1)\oplus \mathcal{O})$. 

Assume that neither $A_{1}$ nor $A_{2}$ are nef divisors of $Y$. As we have already observed, this means that both divisors are exceptional divisors of two elementary divisorial contractions of $Y$ which we denote by $\Psi_{i}\colon Y\to Z_{i}$, for $i=1,2$. Since $A_{1}\cap A_{2}=\emptyset$ we get $\text{Exc}(\Psi_{1})\cap \text{Exc}(\Psi_{2})=\emptyset$.  

Using the main Theorem of \cite{S_WIS} and \cite[Table 3]{MM3} (see also \cite[Table 12.4]{IP99}) we show that in this case the only possible candidate is $Y\cong \mathbb{P}_{\mathbb{P}^{1}\times \mathbb{P}^{1}}(\mathcal{O}(0,-1)\oplus \mathcal{O}(-1,0))$. To this end, we prove that the other varieties with two divisorial contractions in the list given by the main Theorem of \cite{S_WIS} cannot be target of $f$. These varieties correspond to No. 17, 24, 30 in \cite[Table 3]{MM3}. Let us recall that No. 17 is a smooth divisor of tridegree $(1,1,1)$ in $\mathbb{P}^1\times \mathbb{P}^1\times \mathbb{P}^2$, No. 24. is $\mathbb{P}_{\mathbb{F}_1}(\beta^*(T_{\mathbb{P}^2}(-2)))$ and No. 30 is $\mathbb{P}_{\mathbb{F}_1}(\beta^*(\mathcal{O}_{\mathbb{P}^2}\oplus\mathcal{O}_{\mathbb{P}^2}(-1)))$, where $\beta:\mathbb{F}_1\to \mathbb{P}^2$ is the blow-up map. 

Assume by contradiction that $Y$ is isomorphic to No. 30. Using the Mori and Mukai's description of this Fano 3-fold as the blow-up of the strict transform of a line passing through the center of the blow-up map $\mathbb{P}(\mathcal{O}_{\mathbb{P}^2}\oplus\mathcal{O}_{\mathbb{P}^2}(-1))\to \mathbb{P}^3$ (see \cite[Table 3]{MM3}), it is easy to check that $\Exc{(\Psi_{1})}\cap \Exc{(\Psi_{2})}\neq \emptyset$, hence a contradiction. 

Mori and Mukai's classification shows that No. 17 has two elementary divisorial contractions, $\Psi_{i}\colon Y\to \mathbb{P}^{1}\times \mathbb{P}^{2}$ of type $(2,1)$ for $i=1,2$. Assume by contradiction that the target $Y$ is isomorphic to No. 17. This means that $\Exc{(\Psi_1)}\cap \Exc{(\Psi_2)}=\emptyset$. Being $Y$ a Fano bundle of rank $2$ on $\mathbb{F}_1$, we can apply Lemma \ref{lemma1} by taking $\Exc{(\Psi_1)}$ as the divisor $D$ in the Lemma's statement and setting $\psi:=\Psi_2$. In this way we reach a contradiction because $Y\ncong \mathbb{P}^1\times \mathbb{F}_1$. We use the same procedure to exclude the Fano 3-fold No. 24. Indeed, by Mori and Mukai's classification we know that this variety has two elementary divisorial contractions $\Psi_{i}$, where $\Psi_{2}\colon Y\to \mathbb{P}^{1}\times \mathbb{P}^{2}$ of type $(2,1)$. If we assume that the target $Y$ is isomorphic to No. 24 we reach a contradiction by applying Lemma \ref{lemma1} as above, and we get $(b)$. \\ 

Now we show $(c)$. Let us consider the morphism $\xi \circ f \colon X\to S$. Assume that $\rho_{X}\geq 7$, so that $\rho_{S}\geq 3$. By \cite[Theorem 1.1 (i)]{CAS3} one has $X\cong S_{1}\times S$, where $S_{1}$, $S$ are del Pezzo surfaces. Then by Remark \ref{case_product} one has $Y\cong \mathbb{P}^{1}\times S$, and $f$ is induced by a conic bundle $S_{1}\to \mathbb{P}^{1}$. Then we get $(c)$. 

Finally, since $\rho_{X}-\rho_{Y}=3$, then $\rho_{S_{1}}=4$, so that $\rho_{X}\leq 13$. We have already proved that $\rho_{X}\geq 5$, and this complete the proof. 
\end{proof}

As a consequence of the proof of the above proposition we get an explicit description of the discriminant divisor of our conic bundle in many cases. Then we have all ingredients that we need to show Theorem \ref{main_thm}.  

\begin{coro} \label{corollary_prop1} Let $f\colon X\to Y$ be a Fano conic bundle with $\dim{(X)}=4$, and $\rho_{X}-\rho_{Y}=3$. Let us denote by $\bigtriangleup_{f}=A_{1}\sqcup A_{2}$ the discriminant divisor of $f$, with $A_{i}$ smooth components of $\bigtriangleup_{f}$. If $\rho_X=5$ then $A_i\cong \mathbb{P}^{2}$ for some $i=1,2$. If $\rho_X=6$ then either $A_i\cong \mathbb{P}^{1}\times \mathbb{P}^{1}$ or $A_i\cong \mathbb{F}_{1}$ for some $i=1,2$. Moreover, if $Y\cong \mathbb{P}^{1}\times S$ with $S$ a del Pezzo surface such that $\rho_{S}\in \{1,2\}$, then $A_i\cong S$ for each $i=1,2$.
\end{coro}

\begin{proof} Take $\xi\colon Y\to S$ the smooth $\mathbb{P}^{1}$-fibration given by Theorem \ref{thm_tesi} $(b)$, with $S$ a del Pezzo surface. Therefore if $\rho_X=5$ one has $S\cong \mathbb{P}^{2}$, and if $\rho_X=6$ then either $S\cong \mathbb{P}^{1}\times\mathbb{P}^{1}$ or $S\cong \mathbb{F}_{1}$.  We reach the statement by analyzing the positivity of the divisors $A_i$. First we assume that only one between $A_1$ and $A_2$ is a nef divisor, say $A_1$. As we have already observed in the proof of Proposition \ref{prop1}, this implies that $A_2$ is the exceptional divisor of the birational contraction $\varphi:=\varphi_{|mA_1|}\colon Y\to T$ given by the linear system $|mA_1|$ for $m \gg 0$, $m\in \mathbb{N}$, which contracts $A_2$ to a point. More precisely, if $\rho_X=5$ this happens if $Y\cong \mathbb{P}_{\mathbb{P}^{2}}(\mathcal{O}\oplus \mathcal{O}(1))$ or $Y\cong \mathbb{P}_{\mathbb{P}^{2}}(\mathcal{O}\oplus \mathcal{O}(2))$, while if $\rho_X=6$ we are in the case in which $Y\cong \mathbb{P}_{\mathbb{P}^{1}\times \mathbb{P}^{1}}(\mathcal{O}(-1,-1)\oplus \mathcal{O})$ and one can check that the exceptional divisor $A_2$ of $\varphi$ is a section of $\xi\colon Y\to S$ \footnote{We refer the reader to \cite[Lemma 3.9]{CD15} for another more general explanation.}. Assume that both divisors $A_i$ are not nef. Again by the proof of Proposition \ref{prop1} this happens when $Y\cong \mathbb{P}_{\mathbb{P}^{1}\times \mathbb{P}^{1}}(\mathcal{O}(0,-1)\oplus \mathcal{O}(-1,0))$, and the divisors $A_i$ are exceptional divisors of the two different divisorial extremal contractions of $Y$, then $A_i\cong \mathbb{P}^{1}\times \mathbb{P}^{1}$ for each $i=1,2$. Suppose that both $A_i$ are nef divisors. In the proof of Proposition \ref{prop1} we deduced that $Y\cong \mathbb{P}^{1}\times S$ and each $A_i$ is sent to a point through the projection $\mathbb{P}^{1}\times S\to \mathbb{P}^1$, so that $A_i\cong S$ for $i=1,2$. Finally, assume that $Y\cong \mathbb{P}^{1}\times S$. We note that in this latter case both $A_i$ are nef divisors. Indeed, if at least one between $A_i$ is not nef, then $Y$ must be one of the varieties listed above, which are not products. Therefore the claim follows as above.
\end{proof}

\begin{proof}[\bf{\bf{Proof of Theorem \ref{main_thm}}}] The first implication follows by Proposition \ref{first_implication}.

Assume now that $X$ is a Fano 4-fold such that $X\ncong S_{1}\times S_{2}$, where each $S_{i}$ is a del Pezzo surface, and that there exists a conic bundle $f\colon X\to Y$ such that $\rho_{X}-\rho_{Y}=3$. By Proposition \ref{prop1} $(c)$ we are left to study the two possible cases: $\rho_{X}\in \{5,6\}$. Let us denote by $A_{i}$ the two smooth components of the discriminant divisor of $f$, as in Proposition \ref{factorization} $(c)$.

Assume that $\rho_{X}=5$. By Corollary \ref{corollary_prop1} there exists a smooth component $A_i$ of the discriminant divisor $\bigtriangleup_f$ such that $A_{i}\cong \mathbb{P}^{2}$. For simplicity, assume that it is $A_{1}$. Let us consider the divisor $E_{1}$ of $X$ as in Proposition \ref{factorization} $(b)$ such that $E_{1}\to A_{1}$ is a $\mathbb{P}^{1}$-bundle, so that $\rho_{E_{1}}=2$. Since $\dim \mathcal{N}_{1}(E_{1}, X)\leq \rho_{E_{1}}$ it follows that $\delta_{X}\geq 3$. Then $\delta_{X}=3$ by Theorem \ref{prodotto} $(a)$.

Assume that $\rho_{X}=6$. We apply the same argument of the previous case. The only difference here is that by Corollary \ref{corollary_prop1} it follows that there exists a component of the discriminant divisor, say $A_1$, such that either $A_{1}\cong \mathbb{P}^{1}\times \mathbb{P}^{1}$ or $A_{1}\cong \mathbb{F}_{1}$. In any case $\rho_{A_{1}}=2$, and taking the divisor $E_{1}\subset X$ as in Proposition \ref{factorization} $(b)$ we get $\rho_{E_{1}}=3$, and hence the statement follows as above.
\end{proof}

\begin{rema} It is easy to see that Theorem \ref{main_thm} holds in dimension $3$. Assume that $f\colon X\to S$ is a Fano conic bundle with $\dim{(X)}=3$, and $\rho_{X}-\rho_{S}=3$. By \cite[Proposition 4.16]{MM}, $S$ is a del Pezzo surface. Using Theorem \ref{thm_tesi} (b) there exists a smooth $\mathbb{P}^1$-fibration $S\to \mathbb{P}^1$, then either $S\cong \mathbb{P}^1\times \mathbb{P}^1$ or $S\cong \mathbb{F}_1$ and hence $\rho_X=5$. We recall by Proposition \ref{factorization} that the discriminant divisor $\bigtriangleup_{f}$ is given by two smooth disjoint curves $C_{i}$ on $S$, such that $f^{*}(C_i)=E_{i}+\hat{E}_i$ where $E_{i}\to C_{i}$ and $\hat{E}_{i}\to C_{i}$ are $\mathbb{P}^1$-bundle for $i=1,2$. Since $\rho_{C_i}=1$, we get $\rho_{E_i}=2$, and hence $\delta_X\geq 3$. If $X\ncong S\times \mathbb{P}^1$ with $S$ a del Pezzo surface, then by Theorem \ref{prodotto} (a) we deduce that $\delta_X=3$. In particular, if $\rho_{X}\geq 6$ then $X\cong S\times \mathbb{P}^1$ with $S$ a del Pezzo surface (see \cite[Theorem 2]{MM3}). In this case, by Remark \ref{case_product} we know all conic bundle structure on $X$, and $\delta_{X}=\rho_{S}-1$. Moreover, there exists an example of a Fano conic bundle $X\to \mathbb{F}_{1}$ where $X$ is a toric Fano $3$-fold which is not a product, $\rho_{X}=5$ and $\delta_{X}=3$ (see \cite[$\S$4.3.1]{TESI} for details). 
\end{rema}

\begin{rema} Theorem \ref{main_thm} holds for higher dimensions if we can exclude the case $\delta_{X}=2$. This is because using \cite[Lemma 3.10]{IO}, and Theorem \ref{prodotto} (a) it follows that if $f\colon X\to Y$ is a Fano conic bundle with ${\rho_{X}-\rho_{Y}=3}$, and $X\ncong S\times T$ where $S$ is a del Pezzo surface and $T$ is a $(n-2)$-dimensional Fano manifold, then $\delta_{X}\in \{2,3\}$.
\end{rema}

\begin{proof}[\bf{\bf{Proof of Corollary \ref{cor_1}}}] It is enough to prove the statement when $X\ncong S_{1}\times S_{2}$ where $S_{i}$ are del Pezzo surfaces for $i=1,2$. By Theorem \ref{main_thm} the two conditions required in the assumption are equivalent. Take a conic bundle $f\colon X\to Y$ where $\rho_{X}-\rho_{Y}=3$. By Proposition \ref{prop1} we know all possible targets $Y$. In particular, all of these $Y$ are rational varieties. Let us take a factorization for $f$ as in Proposition \ref{factorization} $(a)$, and let us denote by $g\colon X_{2}\to Y$ the elementary conic bundle of this factorization. By Theorem \ref{thm_tesi} $(b)$ $g$ is a smooth conic bundle. By \cite[Proposition 4.3]{FUJ16} it follows that $X_{2}\cong \mathbb{P}_{Y}(\mathcal{F})$, where $\mathcal{F}$ is a rank $2$ vector bundle on $Y$. Hence $X_{2}$ is rational and so is $X$.  
\end{proof}

We conclude this subsection by giving the proof of Theorem \ref{thm5}, in which due to the definition of $\delta_{X}$, and Theorem \ref{main_thm} we find a characterization of Fano $4$-folds admitting a conic bundle with relative Picard number $3$, in terms of the Picard number of prime divisors on $X$.

\begin{proof}[\bf{\bf{Proof of Theorem \ref{thm5}}}] Assume that such an $X$ admits a conic bundle structure $f\colon X\to Y$ with $\rho_{X}-\rho_{Y}=3$. By Proposition \ref{prop1}, one has $\rho_{X}\in \{5,6\}$. In the proof of Theorem \ref{main_thm} we have already observed that there exists a prime divisor $E_1\subset X$ such that $\rho_{E_1}=2$ if $\rho_{X}=5$, and $\rho_{E_1}=3$ if $\rho_{X}=6$, then we get the statement. 
 
Conversely, assume that there exists a prime divisor $D$ on $X$ such that $\rho_{D}=\rho_{X}-3$. Since $\dim{\mathcal{N}_{1}(D,X)}\leq \rho_{D}$, and $X$ is not a product of del Pezzo surfaces, by Theorem \ref{prodotto} $(a)$ we get $\delta_{X}=3$. Hence the statement follows by Theorem \ref{main_thm}.  
\end{proof}

\section{Case of $\rho_{X}=5$}

The goal of this section is to prove Theorem \ref{thm2}, which can be viewed as a first step towards the classification of Fano $4$-folds with $\delta_{X}=3$. To this end, we start focusing on case in which $X$ is a Fano $4$-fold that admits a conic bundle $f\colon X\to Y=\mathbb{P}^{1}\times \mathbb{P}^{2}$ with $\rho_{X}-\rho_{Y}=3$.  

\begin{prop} \label{lemma2} Let $X$ be a Fano 4-fold which admits a smooth $\mathbb{P}^{1}$-fibration $g\colon X\to \mathbb{P}^{1}\times \mathbb{P}^{2}$. Let us consider the contraction $\Psi \colon X\to \mathbb{P}^{1}\times \mathbb{P}^{2} \to \mathbb{P}^{1}$ and let $F_1$ and $F_2$ be two different fibers of $\Psi$. Assume that $F_i\to \mathbb{P}^2$ has a section $B_i$ such that the blow-up of $X$ along $B_i$ is Fano, for $i=1,2$. Then either $X\cong \mathbb{P}^{1}\times Z$, with $Z$ a Fano $\mathbb{P}^{1}$-bundle over $\mathbb{P}^{2}$, or $X$ is the blow-up along a section of the $\mathbb{P}^2$-bundle $\mathbb{P}_{\mathbb{P}^2}(\mathcal{O}\oplus\mathcal{O}(a)\oplus \mathcal{O}(b))$, where $(a,b)\in \{(0,0),(0,1),(0,2),(1,1),(1,2)\}$. 
\end{prop}

\begin{proof}
Let us consider the contraction $\xi\colon \mathbb{P}^{1}\times \mathbb{P}^{2}\to \mathbb{P}^{2}$. Take another extremal contraction $h$ of $X$, different by $g$, and such that $\NE{(h)}\subset \NE{(\xi \circ g)}$. We get another factorization for $\psi:=\xi \circ g\colon X\to \mathbb{P}^{2}$, given by $X\stackrel{h}{\to}Z\stackrel{\xi^{\prime}}{\to}{\mathbb{P}^{2}}$. We have the following commutative diagram:
$$\xymatrix{
& Z \ar[dr]^{\xi^{\prime}} & \\
X \ar[ur]^h \ar[rr]^{\psi} \ar[dr]_g \ar@/_/[ddr]_{\Psi} & & \mathbb{P}^2 \\
& \mathbb{P}^1 \times \mathbb{P}^2 \ar[ur]_{\xi} \ar[d]_{\pi} & \\
& \mathbb{P}^1 &
} $$

Let $R$ be the extremal ray corresponding to $h$. We claim that $R$ is not contracted by $\Psi$ and hence $\Psi$ is finite on the fibers of $h$. To see this, assume by contradiction that $R$ is contracted by $\Psi$. Since a curve in $\mathbb{P}^1 \times \mathbb{P}^2$ cannot be contracted by both projections $\pi$ and $\xi$, if $R$ is contracted by $\Psi$ it follows that it is also contracted by $g$. Hence $g$ is the contraction of $R$, which is absurd by the choice of $h$.

We prove that all fibers of $h$ are at most one-dimensional. By construction every fiber of $h$ is contained in a fiber of $\psi$, then it has dimension at most $2$. Assume by contradiction that there exists a two-dimensional fiber of $h$. Then there is a fiber of $g$ which is contracted by $h$, a contradiction. Therefore, it follows from \cite[Corollary 1.4]{WIS} that $h$ is either a conic bundle or it is the blow-up of $Z$ along a smooth surface. For every $p\in \mathbb{P}^2$, the fiber $S_p:=\psi^{-1}(p)$ is a smooth del Pezzo surface since $\psi$ is a smooth morphism, and $S_p$ admits a smooth $\mathbb{P}^1$-fibration $g|_{S_p}:S_p \to \mathbb{P}^1$. Thus $S_p\cong \mathbb{P}^1\times \mathbb{P}^1$ or $S_p\cong \mathbb{F}_1$. By the deformation invariance of the Fano index (see for instance \cite[Proposition 6.2]{GJ18}), the fibers of $\psi$ are all isomorphic to one another. More precisely, they are all isomorphic to $\mathbb{P}^1\times\mathbb{P}^1$ when $h$ is a conic bundle, while they are isomorphic to $\mathbb{F}_1$ when $h$ is a blow-up.

Assume that $h$ is a conic bundle. Then by \cite[Corollary of Proposition 4.3]{WIS} we know that $Z$ is Fano. Moreover, since all the fibers of $\psi$ are isomorphic to $\mathbb{P}^1\times\mathbb{P}^1$ and $h$ is an equidimensional morphism, it follows that the restriction $h|_{S_p}:S_p \to \mathbb{P}^1$ has to be one of the two conic bundles on $\mathbb{P}^1\times \mathbb{P}^1$ and hence $Z$ is a $\mathbb{P}^1$-bundle over $\mathbb{P}^2$ and the fibers of $h$ are isomorphic to $\mathbb{P}^1$. Since $h:X\to Z$ is a $\mathbb{P}^1$-bundle and $\Psi:X\to \mathbb{P}^1$ is finite over the fibers of $h$, using \cite[Lemma 4.9]{CAS4} we get $X\cong \mathbb{P}^1\times Z$ with $Z$ a Fano $\mathbb{P}^1$-bundle over $\mathbb{P}^2$. From now on, assume that $h$ is birational. Let us denote by $E\subset X$ the exceptional divisor of $h$ and by $A=h(E)\subset Z$ the center of the blow-up. For the reader's convenience the rest of the proof is subdivided into three steps.  

\noindent {\bf Step 1}. We show that $E$ is a section of $g\colon X\to \mathbb{P}^{1}\times \mathbb{P}^{2}$. Moreover, if $B_{1}\subset F_{1}$ and $B_{2}\subset F_{2}$ are sections of the $\mathbb{P}^{1}$-bundles $F_{i}\to \mathbb{P}^2$ as in the assumption, then $E\cap (B_{1}\cup B_{2})=\emptyset$. 
\begin{proof}[\bf{\bf{Proof of Step 1}}] In order to prove that $E$ is a section of $g:X\to \mathbb{P}^1\times \mathbb{P}^2$ we first fix some notation. Given $(t,p)\in \mathbb{P}^1\times \mathbb{P}^2$, set $F_t := \Psi^{-1}(t)$ and $S_p:=\psi^{-1}(p)$. Then $\psi|_{F_t}:F_t\to \mathbb{P}^2$ is a $\mathbb{P}^1$-bundle with fiber $C_{(t,p)}=(\psi|_{F_t})^{-1}(p)\cong \mathbb{P}^1$, and we have that $C_{(t,p)}=F_t \cap S_p$.

As we have already observed $S_p\cong \mathbb{F}_1$ for every $p\in \mathbb{P}^2$, then $g|_{S_p}:S_p\to \mathbb{P}^1\times \{p\}$ corresponds to the $\mathbb{P}^1$-bundle morphism of $\mathbb{F}_1$, and $h|_{S_p}:S_p\to h(S_p)= \mathbb{P}^2$ is the contraction of the exceptional curve $E_p\subset S_p\cong \mathbb{F}_1$. 
It is easy to check that $E_{p}\subset E$ for each $p\in \mathbb{P}^{2}$, and since all fibers $S_{p}$ are different, we obtain that $\{E_{p}\}_{p\in \mathbb{P}^{2}}$ is a family of $(-1)$-curves which covers $E$. 

Take $\Gamma$ a fiber of $g$. We prove that $E\cdot \Gamma=1$. We observe that $\Gamma\subset F_{t}$ for some $t\in \mathbb{P}^{1}$, and by what we have already observed $S_{p}=\psi^{-1}(\psi(\Gamma))$ is such that $F_{t}\cap S_{p}\cong \Gamma$. Using the intersection theory of $\mathbb{F}_1\cong S_{p}$ one has $E_{p}\cdot \Gamma=1$. Since $\{E_{p}\}_{p\in \mathbb{P}^{2}}$ covers $E$, we deduce that $E$ is a section of $g$.

Now we show that the divisor $E$ is disjoint from the sections $B_1$ and $B_2$. We know that $h$ contracts the divisor $E$ onto the surface $A$, and $E$ is covered by curves of anticanonical degree 1 (cf. \cite[Lemma 1]{ISHII}). Using the above information and the assumption that the blow-up of $X$ along $B_1$ is Fano we observe that whenever $B_1$ intersects a fiber of $h$, this fiber must be contained in $B_1$. Indeed, let us denote by $\ell$ a non-trivial fiber of $h$, by $\varphi_1\colon X_1\to X$ the blow-up of $X$ along $B_1\subset X$ with exceptional divisor $E_1\subset X_1$ and by $\widetilde{\ell}\subset X_1$ the strict transform of $\ell$ in $X_1$. If $\ell$ intersects $B_1$ but $\ell \not\subset B_1$ we have that $E_1\cdot \widetilde{\ell}>0$ and hence
$$-K_{X_1}\cdot \widetilde{\ell}=\varphi_1^*(-K_X)\cdot \widetilde{\ell}-E_1\cdot \widetilde{\ell}=-K_X\cdot \ell - E_1 \cdot \widetilde{\ell} = 1 - E_1 \cdot \widetilde{\ell} \leq 0 $$
which is impossible if $X_1$ is Fano.

From the above observation we deduce that $E\cap B_1 = \emptyset$. Indeed, since $\Psi$ is finite on fibers of $h$ and $B_1\subset F_1$, $B_1$ cannot contain a fiber of $h$ and hence the fibers of $h$ are disjoint from $B_1$. We deduce by the same argument that $E\cap B_2=\emptyset$ and hence $E\cap (B_{1}\cup B_{2})=\emptyset$.
\end{proof}

\noindent {\bf Step 2}. We show that $\xi^{\prime}\colon Z\to \mathbb{P}^2$ is a $\mathbb{P}^2$-bundle and $A$, $h(B_1)$, $h(B_2)$ are mutually disjoint sections of $\xi^{\prime}$.  
\begin{proof}[\bf{\bf{Proof of Step 2}}] We know that for every $p\in \mathbb{P}^2$, $S_p=\psi^{-1}(p)\cong \mathbb{F}_1$, and $g_{\mid S_p}\colon S_p\to \mathbb{P}^1\times \{p\}$ is the $\mathbb{P}^1$-bundle morphism of $\mathbb{F}_1$. From the definition of $h$ it follows that $h_{\mid S_p}\colon S_p\to \mathbb{P}^2$ is the blow-down, so that ${\xi^{\prime}}^{-1}(p)=h(S_p)= \mathbb{P}^2$. This shows that the scheme-theoretic intersection $A\cap {\xi^{\prime}}^{-1}(p)$ consists of a single point. Thus $\xi^{\prime}_{\mid A}\colon A\to \mathbb{P}^2$ is birational. In Step 1 we have already proved that $E$ is a section of $g$, then $\rho_E=2$. Since $A$ is smooth we see that $E$ is a $\mathbb{P}^1$-bundle over $A$, so that $\rho_A=1$. This shows that $\xi^{\prime}_{\mid A}\colon A\to \mathbb{P}^2$ is an isomorphism, then $A$ is a section of $\xi^{\prime}$. Since $B_i$ is a section of $\psi$ for $i=1,2$, and by Step 1 we know that $E\cap (B_{1}\cup B_{2})=\emptyset$, it follows that also each $h(B_i)$ is a section of $\xi^{\prime}$, then the claim. 
\end{proof}

\noindent {\bf Step 3}. \label{Step7} We show that $Z\cong \mathbb{P}_{\mathbb{P}^{2}}(\mathcal{O}\oplus \mathcal{O}(a)\oplus \mathcal{O}(b))$, where $(a,b)\in \{(0,0),(0,1),(0,2),$ $ (1,1), (1,2)\}$.  
\begin{proof}[\bf{\bf{Proof of Step 3}}]
By Step 2 it follows that $Z\cong \mathbb{P}_{\mathbb{P}^{2}}(\mathcal{O}\oplus \mathcal{O}(a)\oplus \mathcal{O}(b))$ with $a,b \in \mathbb{Z}$. If $Z$ is Fano then by Batyrev-Sato classification of toric Fano 4-folds \cite{BAT99,SATO} we get $(a,b)\in \{(0,0),(0,1),(0,2), (1,1)\}$, while if $Z$ is not Fano using \cite[Proposition 3.6]{WIS} we obtain that $(a,b)=(1,2)$.
\end{proof}
Hence the statement follows. 
\end{proof}

\begin{lemm} \label{lemma3} Let $f\colon X\to \mathbb{P}^{1}\times \mathbb{P}^{2}$ be a Fano conic bundle with $\rho_{X}=5$. Then there exists a factorization as in Proposition \ref{factorization} $(a)$, where $X_{1}$, $X_{2}$ are Fano $4$-folds. 
\end{lemm}

\begin{proof} Let us take a factorization for $f$ as in Proposition \ref{factorization} $(a)$, given by:
$$\xymatrix{X \ar[r]^{\widetilde{f_1}} \ar@/^2.5pc/[rrr]^{f} & \widetilde{X_1} \ar[r]^{\widetilde{f_2}} & \widetilde{X_2} \ar[r]^{\widetilde{g}} & \mathbb{P}^{1}\times \mathbb{P}^{2}}$$

We denote by $A_{i}$ the two components of the discriminant divisor of $f$ as in Proposition \ref{factorization} $(c)$. By Proposition \ref{factorization} $(b)$ we know that there exist two prime divisors $E_{i}$, $\hat{E}_{i}$ of $X$ such that $f^{*}(A_{i})=E_{i}+\hat{E}_{i}$, and $E_{i}\to A_{i}$, $\hat{E}_{i}\to A_{i}$ are $\mathbb{P}^{1}$-bundles for $i=1,2$. Let us denote by $e_{i}\subset E_{i}$, $\hat{e_{i}}\subset \hat{E}_{i}$ the fibers for $i=1,2$. 

By the conic bundle structure, $D_{i}:=\widetilde{g}^{*}(A_{i})$ are $\mathbb{P}^{1}$-bundles over $A_{i}$, and each $D_{i}$ contains a section $A_{i}^{\prime}\cong A_{i}$, such that $f_{i}$ is the blow-up of $X_{i}$ along $A_{i}^{\prime}$ for $i=1,2$. 

By \cite[Remark 3.7]{IO} the factorization of $f$ is not unique and depends on the choice of the extremal rays spanned by $[e_{i}]$, $[\hat{e}_{i}]$. To fix the notation, assume that $\widetilde{f_{1}}$ corresponds to the contraction of the extremal ray $R_{1}=\mathbb{R}_{\geq 0}[e_{1}]$, and $\widetilde{f_{2}}$ is the contraction given by the extremal ray $R_{2}=\mathbb{R}_{\geq 0}[e_{2}]$. Using these facts, we first prove that up to changing the factorization of $f$ we obtain that the target of the first blow-up of the factorization of $f$ is a Fano variety. 

Assume that $\widetilde{X_{1}}$ is not Fano. Then there exists an irreducible curve $C\subset \widetilde{X_{1}}$ such that $-K_{\widetilde{X_{1}}}\cdot C\leq 0$, and $C\subset A_{1}^{\prime}$. We show that $D_{1}$ is not Fano. To this end, we first prove that $D_{1}$ is a nef divisor which has intersection zero with every curve contained in it. Indeed, by Corollary \ref{corollary_prop1} we deduce that the divisors $A_i$ are both fibers of the projection $\mathbb{P}^{1}\times \mathbb{P}^{2}\to \mathbb{P}^{1}$, then $A_1\equiv a A_2$ for some $a\in \mathbb{R}$. By the injectivity of $\widetilde{g}^{*}$ we obtain that $D_1\equiv a D_2$. Therefore, if $C_1\subset D_1$ is an irreducible curve, then $D_1\cdot C_1=a D_2\cdot C_1 = 0$ because $D_1\cap D_2=\emptyset$. Hence by the adjunction formula we get $-K_{D_{1}}\cdot C=-K_{\widetilde{X_{1}}}\cdot C-D_{1}\cdot C\leq 0$. Since $\hat{E}_{1}$ is the strict transform in $X$ of the divisor $D_{1}$, it follows that $\hat{E}_{1}$ is not a Fano variety. 

By Corollary \ref{corollary_prop1}, we know that $A_{i}^{\prime}\cong \mathbb{P}^{2}$, so that $\rho_{\hat{E}_{1}}=2$. Now we replace $\widetilde{f_{1}}$ with the contraction of the extremal ray spanned by $[\hat{e_{1}}]$. In this way we get another blow-up $f_{1}\colon X\to X_{1}$ where $X_{1}$ is Fano by \cite[Proposition 3.4]{WIS}. 

Let us consider the factorization of $f$ given by $X\stackrel {f_{1}}{\to} X_{1} \stackrel{\widetilde{f_{2}}}{\to} \widetilde{X_{2}}\stackrel{\widetilde{g}}{\to} \mathbb{P}^{1}\times \mathbb{P}^{2}$. 

If $\widetilde{X_{2}}$ is Fano we are done. Assume that it is not. We can proceed in the same way, replacing $A_{1}^{\prime}$ with $A_{2}^{\prime}$, which is the center of the blow-up $\widetilde{f_{2}}$. With a similar method, it is easy to check that $D_{2}$ is not Fano, and since its strict transform in $X_{1}$ is isomorphic to $\hat{E_{2}}$, we deduce that $\hat{E_{2}}$ is not Fano. 

Replacing $\widetilde{f_{2}}$ with $f_{2}\colon X_{1} \to X_{2}$ which corresponds to the contraction of the extremal ray spanned by $[\hat{e}_{2}]$, we complete the proof applying again \cite[Proposition 3.4]{WIS}. 
\end{proof}

\begin{rema} The proof of Lemma \ref{lemma3} depends neither on the dimension of $X$ nor on $\rho_{X}$. Indeed, the same argument shows the following more general statement. Assume that $X$ is a Fano manifold of arbitrary dimension, and let $f\colon X\to Y$ be a non-elementary Fano conic bundle with $r:=\rho_{X}-\rho_{Y}$. Let us take a factorization for $f$ as in Proposition \ref{factorization} $(a)$, and let us denote by $g\colon X_{r-1}\to Y$ the elementary conic bundle in this factorization. If $\rho_{A_{i}}=1$, and $D_{i}:=g^{*}(A_{i})$ is a nef divisor for $i=1,\dots,r-1$, then there exists a factorization of $f$ as in Proposition \ref{factorization} $(a)$, where each of the $X_{i}$ is a Fano variety.  
\end{rema}

\begin{proof}[\bf{\bf{Proof of Theorem \ref{thm2}}}] Let us take a factorization for $f$ as in Proposition \ref{factorization} (a). We denote by $g\colon X_{2}\to Y$ the associated elementary conic bundle, and by $A_{i}$ the two disjoint components of $\bigtriangleup_{f}$. By Proposition \ref{prop1} we know all possible $Y$. First we assume that $Y\cong \mathbb{P}^{1}\times \mathbb{P}^{2}$ and, by Corollary \ref{corollary_prop1}, we deduce that in this case $A_{i}\cong \mathbb{P}^2$ for each $i=1,2$. Applying Lemma \ref{lemma3} to $f\colon X\to Y\cong \mathbb{P}^{1}\times \mathbb{P}^{2}$, and keeping the same notation of this lemma, we can suppose that $X_1$ and $X_{2}$ are Fano. We observe that the contraction $\Psi\colon X_2 \stackrel{g}{\to} \mathbb{P}^{1}\times \mathbb{P}^{2}\to \mathbb{P}^{1}$ satisfies the assumptions of Proposition \ref{lemma2}. To this end, we first observe that $g$ is smooth by Theorem \ref{thm_tesi} $(b)$. Moreover, by the proof of \cite[Proposition 3.5]{IO} it follows that each $\mathbb{P}^{1}$-bundle $F_i:=g_{\mid g^{-1}(A_i)}\colon g^{-1}(A_i)\to A_i$ has a section $B_i\cong A_i$, and such $B_i$ corresponds to the center of the blow-ups $f_i$ in the factorization of $f$. Hence we can apply Proposition \ref{lemma2} to the contraction $\Psi$, and taking the two $\mathbb{P}^{1}$-bundles $F_i$ which are fibers of $\Psi$. Then it follows that either $X_{2}$ is isomorphic to one of the varieties as in point $(1)$ or that $X_2$ is isomorphic to the blow-up of a section $A$ of the $\mathbb{P}^2$-bundle $Z\cong \mathbb{P}_{\mathbb{P}^2}(\mathcal{O}\oplus\mathcal{O}(a)\oplus\mathcal{O}(b))$, where $(a,b)\in \{(0,0),(0,1),(0,2),(1,1),(1,2)\}$. By the proof of Proposition \ref{lemma2} we recall that this last case occurs when there is a birational map $h\colon X_{2}\to Z$ which is a blow-up along a section $A$ of $Z$, such that $A\cap (h(B_{1})\cup h(B_{2}))=\emptyset$. Moreover, by Step 1 in the proof of Proposition \ref{lemma2}, it follows that $h(B_{i})\cong B_{i}$ for $i=1,2$, and by Step 2 of the same Proposition we know that $A$, $B_{1}$, $B_{2}$ are disjoint sections of $Z$.  Therefore using the factorization of $f$ we conclude that $X$ is the blow-up along these three disjoint sections of $Z$. 

Finally, assume that $Y\cong \mathbb{P}_{\mathbb{P}^{2}}(\mathcal{O}\oplus \mathcal{O}(1))$ or $Y\cong \mathbb{P}_{\mathbb{P}^{2}}(\mathcal{O}\oplus \mathcal{O}(2))$. Being $g\colon X_{2}\to Y$ a smooth $\mathbb{P}^{1}$-fibration, and $Y$ rational, then applying \cite[Proposition 4.3]{FUJ16} we deduce that $X_{2}$ is as in point $(2)$ and hence the statement.  
\end{proof}

\section{Toric Fano conic bundles with relative Picard dimension $3$}

The purpose of this section is to provide examples proving that all Fano 3-folds $Y$ listed in Proposition \ref{prop1} appear as target of some Fano conic bundle $f\colon X\to Y$ with $\rho_X-\rho_Y=3$. In particular, we give explicit examples of toric Fano conic bundle, \textit{i.e.} Fano conic bundle $f\colon X\to Y$ where $X$ is toric. 

We refer the reader to \cite{CLS11} for the general theory of toric varieties and to \cite{FUJ03,FS04,MAT02,REID} for details on the toric MMP. 

Let $N\cong \mathbb{Z}^n$ be a lattice and let $N_\mathbb{R}\cong \mathbb{R}^n$ be its real scalar extension. Let $\Delta_X\subseteq N_\mathbb{R}$ be a fan and let us denote by $\Delta_X(k)$ the set of $k$ dimensional cones in $\Delta_X$ and by $X=X(\Delta_X)$ the associated toric variety. If $\sigma\in \Delta_X(k)$ is a cone generated by the primitive vectors $\{u_1,\ldots,u_k\}$ we will denote by $V(\sigma)$ or $V(u_1,\ldots,u_k)$ the closed invariant subvariety of codimension $k$ in $X$. In particular, each $\rho\in \Delta_X(1)$ corresponds to an invariant prime divisor $V(\rho)$ on $X$; such a cone is called a {\it ray}. Similarly, each cone of codimension one $\omega\in \Delta_X(n-1)$ corresponds to an invariant rational curve on $X$; such a cone is called a {\it wall}. 

If $X$ is a smooth projective toric variety of dimension $n$ then every extremal ray $R\subseteq \operatorname{NE}(X)$ corresponds to an invariant curve $C_\omega$ such that $R=\mathbb{R}_{\geq 0}[C_\omega]$ or, equivalently, to a wall $\omega\in \Delta_X(n-1)$.

Let us suppose that such a wall $\omega$ is generated by the primitive vectors $\{u_1,\ldots,u_{n-1}\}$. Since the fan $\Delta_X$ is simplicial, $\omega$ separates two maximal cones $\sigma=\operatorname{cone}(u_1,\ldots,u_{n-1},u_n)$ and $\sigma'=\operatorname{cone}(u_1,\ldots,u_{n-1},u_{n+1})$, where $u_n$ and $u_{n+1}$ are primitive on rays on opposite sides of $\omega$. The $n+1$ vectors $u_1,\ldots,u_{n+1}$ are linearly dependent. Hence, they satisfy a so called {\it wall relation} (see \cite[page 303]{CLS11}):
$$b_nu_n + \sum_{i=1}^{n-1}b_iu_i+b_{n+1}u_{n+1}=0, $$ 
where $b_n=b_{n+1}=1$ and $b_i\in \mathbb{Z}$ for $i=1,\ldots,n-1$. By reordering if necessary, we can assume that
\begin{center}
$\begin{array}{ccc}
b_i<0 & \text{for} & 1\leq i \leq \alpha \\
b_i=0 & \text{for} & \alpha+1\leq i \leq \beta \\
b_i>0 & \text{for} & \beta+1\leq i \leq n+1.
\end{array}$ 
\end{center}
The wall relation and the signs of the coefficients involved allow us to describe the nature of the associated contraction. More precisely, 
\begin{itemize}
 \item[(1)] Fiber type contractions correspond to $\alpha=0$
 \item[(2)] Divisorial contractions correspond to $\alpha=1$
 \item[(3)] Small contractions correspond to $\alpha>1$.
\end{itemize}

In all cases we have that the extremal contraction associated to $R$ is of type $(n-\alpha,\beta)$. We refer the reader to \cite[\S 2]{REID} or \cite[\S 14.2]{MAT02} for details. Moreover, toric Fano 4-folds were classified by Batyrev and Sato into 124 families by using the language of {\it primitive collections} \cite{BAT91,BAT99,SATO}. 

In our analysis, we are mainly interested in extremal $K-$negative contractions between smooth projective varieties $\varphi_R\colon X\to X_R$ of type $(n-1,n-2)$ and type $(n,n-1)$ (\textit{i.e.}, an elementary conic bundle). In the toric case, the associated wall relations are the following:
\begin{itemize}
 \item[(a)] Type $(n-1,n-2)$: $u_n - u_1 + u_{n+1} = 0$. It corresponds geometrically to the blow-up of the codimension two subvariety $A=V(u_n,u_{n+1})\subset X_R$, with exceptional divisor $V(u_1)\subset X$. The toric morphism $\varphi_R:X\to X_R$ corresponds to the star subdivision of $\Delta_{X_R}$ relative to the ray generated by the vector $u_1 = u_n + u_{n+1}$ (see \cite[Definition 3.3.7]{CLS11}).
 \item[(b)] Type $(n,n-1)$: $u_n + u_{n+1} = 0$. It corresponds geometrically to a projective vector bundle structure $X \cong \mathbb{P}_{X_R}(\mathcal{E})$ where $\mathcal{E}$ is a decomposable rank $2$ vector bundle on $X_R$, for which $V(u_n)\cong V(u_{n+1})\cong X_R$ are disjoint sections of $X$. The toric morphism $\varphi_R: X\to X_R$ is induced by the quotient map ${\Phi_R: N_\mathbb{R} \to N_\mathbb{R}\slash \operatorname{Span}(u_n)}$.
\end{itemize}

Let us recall that in the Fano conic bundle case $f\colon X\to Y$ with $\rho_X-\rho_{Y}= 3$, by Proposition \ref{factorization} $(a)$ and Theorem \ref{thm_tesi} $(b)$ it follows that there exists a factorization of $f$
$$\xymatrix{X \ar[r]^{f_1} \ar@/^2pc/[rrr]^{f} & X_1 \ar[r]^{f_2} & X_2 \ar[r]^{g} & Y}$$
\noindent where each $f_{i}$ is a blow-up along a smooth surface of $X_{i}$ for $i=1,2$, and $g$ is a smooth elementary conic bundle. By \cite[Proposition 3.5 (1)]{IO} the centers of the blow-ups are isomorphic to the smooth components of the discriminant divisor of $f$. When $X$ is a Fano 4-fold some of the possible blow-up centers are described in Corollary \ref{corollary_prop1}. 

We denote by $\operatorname{Bl}_{A_i}(X_i)$ the blow-up of $X_{i}$ along $A_{i}$. From now on we keep the above notation for a factorization of $f$ and we use Batyrev-Sato notation as in \cite[\S 4]{BAT99} and \cite[Table 1]{SATO}.

\subsection*{Case of $\rho_X=5$.} 

For a toric Fano conic bundle $f\colon X\to Y$ with $\rho_X = 5$ and $\rho_{X}-\rho_{Y}=3$, we know by Proposition \ref{prop1} $(a)$ that $Y\cong \mathbb{P}_{\mathbb{P}^2}(\mathcal{O}\oplus \mathcal{O}(a))$ for some $a\in \{0,1,2\}$. In particular $X$ admits a locally trivial toric bundle over $\mathbb{P}^2$ whose fiber is a Del Pezzo surface with Picard rank 4. Toric Fano 4-folds admitting such a fibration onto $\mathbb{P}^2$ were considered by Batyrev in \cite[\S 3.2.9]{BAT99} and their combinatorial type is denoted by $K$ in \cite[\S 4]{BAT99}. More precisely, we have the following:

\begin{prop}
Let $X$ be a toric Fano 4-fold with $\rho_X = 5$. Assume that there exists a conic bundle $f\colon X\to Y$ such that $Y\cong \mathbb{P}_{\mathbb{P}^2}(\mathcal{O}\oplus \mathcal{O}(a))$ for some $a\in \{0,1,2\}$. Then $X \cong K_i$ for some $i\in \{1,2,3,4\}$.
\end{prop}

\begin{exem}\label{Example: Toric Picard 5}
The following examples show that all the varieties $K_1, K_2, K_3, K_4$ admit a conic bundle $f\colon X\to Y$ onto $Y\cong \mathbb{P}_{\mathbb{P}^2}(\mathcal{O}\oplus \mathcal{O}(a))$ for some $a\in \{0,1,2\}$. 

Let us denote by $\{e_1,e_2,e_3,e_4\}$ the canonical basis of $N_\mathbb{R}\cong \mathbb{R}^4$. Let us explain in more detail the method used to provide the examples.

We start by considering suitable toric Fano 4-folds $X_2$ with $\rho_X - \rho_{X_2} = 2$ such that there is an elementary smooth conic bundle $g\colon X_2\to Y$, where $Y$ is one of the above listed varieties\footnote{This can be easily checked by looking at the corresponding fans. We refer the interested reader to the auxiliary files accompanying the Macaulay2 package \url{NormalToricVarieties} for the complete list of fans of the 124 toric Fano 4-folds in Batyrev-Sato classification.}. Afterwards, we have to blow-ups two suitable smooth surfaces from $X_{2}$, which in this case we choose to be both isomorphic to $\mathbb{P}^2$ (c.f. Corollary \ref{corollary_prop1}). 

Let us denote by $A_{1}$, $A_{2}$ these centers. By the conic bundle structure, $A_{i}$'s must be transversal with respect to the fibers of the elementary conic bundle $g\colon X_2\to Y$. To this end, it can be checked by using Macaulay2 \cite{M2} whether the blow-up $X$ along $A_1, A_2$ is Fano or not. Then by computing the self-intersection number $(-K_X)^4$ and comparing with \cite[\S 4]{BAT99} and \cite[Table 1]{SATO} we can determine the corresponding varieties $X_1$ and $X$ in Batyrev-Sato classification.
\begin{enumerate}
\item Let $X_2:= D_5 \cong \mathbb{P}^1\times \mathbb{P}_{\mathbb{P}^2}(\mathcal{O}\oplus \mathcal{O}(2))$. The primitive generators of the rays of the fan of $X_2$ are given by
 \begin{equation*}\begin{split}\Delta_{X_2}(1)=\{& u_1 = e_1, u_2 = e_2, u_3 = -e_1-e_2+2e_3,\\ & u_4=e_3,u_5=-e_3,u_6 = e_4,u_7=-e_4 \}.\end{split}\end{equation*}                                                                                                                  
  The elementary conic bundle $g\colon X_2 \to  \mathbb{P}^1 \times \mathbb{P}^2$ induced by the second factor $\mathbb{P}_{\mathbb{P}^2}(\mathcal{O}\oplus \mathcal{O}(2))\to \mathbb{P}^2$ is given by the quotient map ${G\colon \mathbb{R}^4 \to \mathbb{R}^4\slash \langle e_3 \rangle}$. Let $A_1 = V(u_4,u_6)$ and $A_2=V(u_5,u_7)$, where $A_i \cong \mathbb{P}^2$ for $i=1,2$. We verify with Macaulay2 that $X_1:= \operatorname{Bl}_{A_1}(X_2)\cong H_3$ and $X:=\operatorname{Bl}_{A_1,A_2}(X_2)\cong K_1$ are Fano. We get the following factorization of $f:X\to Y$.
  $$X\cong K_1 \xrightarrow{f_1} X_1 \cong H_3 \xrightarrow{f_2} X_2 \cong D_5 \xrightarrow{g} Y \cong \mathbb{P}^1 \times \mathbb{P}^2.$$
 
 \item Let $X_2 := D_3 \cong \mathbb{P}_Y(\mathcal{E})$, where $\mathcal{E}$ is a decomposable rank 2 vector bundle on $Y\cong \mathbb{P}_{\mathbb{P}^2}(\mathcal{O}\oplus \mathcal{O}(1))$. The primitive generators of the rays of the fan of $X_2$ are given by
 \begin{equation*}\begin{split}\Delta_{X_2}(1)=\{& u_1 = e_1, u_2 = e_2, u_3 =-e_1-e_2+e_3+e_4, \\ & u_4=e_3,u_5=-e_3+e_4,u_6 = e_4,u_7=-e_4 \}. \end{split}\end{equation*}
 The elementary conic bundle $g:D_3 \to \mathbb{P}_{\mathbb{P}^2}(\mathcal{O}\oplus \mathcal{O}(1))$ is induced by the quotient map ${G:\mathbb{R}^4 \to \mathbb{R}^4\slash \langle e_4 \rangle}$. Let $A_1 = V(u_4,u_7)$ and $A_2=V(u_5,u_7)$, where $A_i \cong \mathbb{P}^2$ for $i=1,2$. We verify with Macaulay2 that $X_1:=\operatorname{Bl}_{A_1}(X_2)\cong H_2$ and $X:=\operatorname{Bl}_{A_1,A_2}(X_2)\cong K_2$ are Fano. We get the following factorization of $f\colon X\to Y$.
 $$X\cong K_2 \xrightarrow{f_1} X_1 \cong H_2 \xrightarrow{f_2} X_2 \cong D_3 \xrightarrow{g} Y\cong \mathbb{P}_{\mathbb{P}^2}(\mathcal{O}\oplus \mathcal{O}(1)).$$
 
 \item Let $X_2 := D_{16} \cong \mathbb{P}_Y(\mathcal{E})$, where $\mathcal{E}$ is a decomposable rank 2 vector bundle on $Y\cong \mathbb{P}_{\mathbb{P}^2}(\mathcal{O}\oplus \mathcal{O}(2))$. The primitive generators of the rays of the fan of $X_2$ are given by
 \begin{equation*}\begin{split}\Delta_{X_2}(1)=\{ & u_1 = e_1, u_2 = e_2, u_3 = -e_1-e_2+e_3-e_4, \\ & u_4=e_3,u_5=-e_3+e_4,u_6 = e_4,u_7=-e_4 \}. \end{split}\end{equation*}
 The elementary conic bundle $g:D_{16} \to \mathbb{P}_{\mathbb{P}^2}(\mathcal{O}\oplus \mathcal{O}(2))$ is induced by the quotient map ${G:\mathbb{R}^4 \to \mathbb{R}^4\slash \langle e_4 \rangle}$. Let $A_1 = V(u_4,u_7)$ and $A_2=V(u_5,u_7)$, where $A_i \cong \mathbb{P}^2$ for $i=1,2$. We verify with Macaulay2 that $X_1:=\operatorname{Bl}_{A_1}(X_2)\cong H_5$ and $X:=\operatorname{Bl}_{A_1,A_2}(X_2)\cong K_3$ are Fano. We get the following factorization of $f:X\to Y$.
 $$X\cong K_3 \xrightarrow{f_1} X_1 \cong H_5 \xrightarrow{f_2} X_2 \cong D_{16} \xrightarrow{g} Y\cong \mathbb{P}_{\mathbb{P}^2}(\mathcal{O}\oplus \mathcal{O}(2)).$$
 
 \item The conic bundle structure on $X:= K_4 \cong  \operatorname{Bl}_{p_1,p_2,p_3}(\mathbb{P}^2) \times \mathbb{P}^2$ is induced by the first factor as follows (see Remark \ref{case_product}).
 $$K_4 \cong  \operatorname{Bl}_{p_1,p_2,p_3}(\mathbb{P}^2) \times \mathbb{P}^2 \xrightarrow{f_1} X_1 \cong  \operatorname{Bl}_{p_1,p_2}(\mathbb{P}^2) \times \mathbb{P}^2 \xrightarrow{f_2} X_2 \cong \mathbb{F}_1 \times \mathbb{P}^2  \xrightarrow{g} Y\cong \mathbb{P}^1 \times \mathbb{P}^2.$$
\end{enumerate}
\end{exem}

\subsection*{Case of $\rho_X=6$.}

Given a Fano conic bundle $f\colon X\to Y$ with $\rho_X = 6$ and $\rho_{X}-\rho_{Y}=3$, we know all possible targets $Y$ using Proposition \ref{prop1} $(b)$.  The following examples show that all of these Fano 3-folds $Y$ appears as targets of some toric Fano conic bundle $f\colon X\to Y$ with $\rho_X-\rho_Y = 3$.  

\begin{exem}\label{Example: Toric Picard 6}
We proceed as in Examples \ref{Example: Toric Picard 5}. Let us denote by $\{e_1,e_2,e_3,e_4\}$ the canonical basis of $N_\mathbb{R}\cong \mathbb{R}^4$. 
\begin{enumerate}
 \item Let $X_2:=L_1 \cong \mathbb{P}_{\mathbb{P}^1 \times \mathbb{P}^1 \times \mathbb{P}^1}(\mathcal{O}\oplus \mathcal{O}(1,1,1))$. The primitive generators of the rays of the fan of $X_2$ are given by
 \begin{equation*}\begin{split}\Delta_{X_2}(1)=\{ & u_1 = e_1, u_2 = e_2, u_3 = e_1-e_2, u_4=e_3,\\ & u_5=e_1 - e_3,u_6 = e_4,u_7=e_1 - e_4,u_8=-e_1 \}. \end{split}\end{equation*}
 The elementary conic bundle $g:L_1 \to\mathbb{P}^1 \times \mathbb{P}^1 \times \mathbb{P}^1$ is induced by the quotient map ${G:\mathbb{R}^4 \to \mathbb{R}^4\slash \langle e_1 \rangle}$. Let $A_1=V(u_2,u_8), A_2=V(u_3,u_8)$, where $A_i \cong \mathbb{P}^1\times \mathbb{P}^1$ for $i=1,2$.  We verify with Macaulay2 that $X_1:=\operatorname{Bl}_{A_3}(X_1) \cong \operatorname{Bl}_{A_2}(X_2)\cong Q_3$ and $X:=\operatorname{Bl}_{A_1,A_2}(X_2)\cong U_1$ are Fano. We get the following factorization of $f:X\to Y$.
 $$X\cong U_1 \xrightarrow{f_1} X_1 \cong Q_3 \xrightarrow{f_2} X_2 \cong L_{1} \xrightarrow{g} Y \cong \mathbb{P}^1 \times \mathbb{P}^1 \times \mathbb{P}^1. $$
 
 \item Let $X_2:=L_2\cong \mathbb{P}_Y(\mathcal{E})$, where $\mathcal{E}$ is a decomposable rank 2 vector bundle on $Y\cong \mathbb{P}_{\mathbb{P}^1 \times \mathbb{P}^1}(\mathcal{O}(-1,-1)\oplus \mathcal{O})$. The primitive generators of the rays of the fan of $X_2$ are given by
 \begin{equation*}\begin{split}\Delta_{X_2}(1)=\{ & u_1 = e_1, u_2 = e_2, u_3 = e_1 - e_2, u_4=e_3,\\ & u_5= e_1 - e_2 - e_3,u_6 = e_4,u_7=e_1 - e_2 - e_4,u_8=-e_1 \}. \end{split}\end{equation*}
 The elementary conic bundle $g:L_2 \to \mathbb{P}_{\mathbb{P}^1 \times \mathbb{P}^1}(\mathcal{O}(-1,-1)\oplus \mathcal{O})$ is induced by the quotient map $G:\mathbb{R}^4 \to \mathbb{R}^4\slash \langle e_1 \rangle$. Let $A_1=V(u_2,u_8), A_2=V(u_3,u_8)$, where $A_i \cong \mathbb{P}^1\times \mathbb{P}^1$ for $i=1,2$. We verify with Macaulay2 that $X_1:=\operatorname{Bl}_{A_1}(X_2)\cong Q_{13}$ and $X:=\operatorname{Bl}_{A_1,A_2}(X_2)\cong U_1$ are Fano.  We get the following factorization of $f:X\to Y$.
 $$X\cong U_1 \xrightarrow{f_1} X_1 \cong Q_{13} \xrightarrow{f_2} X_2 \cong L_{2} \xrightarrow{g} Y\cong \mathbb{P}_{\mathbb{P}^1 \times \mathbb{P}^1}(\mathcal{O}(-1,-1)\oplus \mathcal{O}).$$
 
 \item Let $X_2:=L_3\cong \mathbb{P}_Y(\mathcal{E})$, where $\mathcal{E}$ is a decomposable rank 2 vector bundle on $Y\cong \mathbb{F}_1 \times \mathbb{P}^1$. The primitive generators of the rays of the fan of $X_2$ are given by
 \begin{equation*}\begin{split}\Delta_{X_2}(1)=\{ & u_1 = e_1, u_2 = e_2, u_3 = e_1-e_2, u_4=e_3,\\ & u_5=e_1-e_3,u_6 = e_4,u_7=e_3-e_4,u_8=-e_1 \}. \end{split}\end{equation*}
 The elementary conic bundle $g\colon L_3 \to \mathbb{F}_1 \times \mathbb{P}^1$ is induced by the quotient map $G:\mathbb{R}^4 \to \mathbb{R}^4\slash \langle e_1 \rangle$. Let $A_1 = V(u_2,u_8), A_2=V(u_3,u_8)$, where $A_i \cong \mathbb{F}_1$ for $i=1,2$. We verify with Macaulay2 that $X_1:=\operatorname{Bl}_{A_1}(X_2)\cong Q_{5}$ and $X:=\operatorname{Bl}_{A_1,A_2}(X_2)\cong U_2$ are Fano. We get the following factorization of $f:X\to Y$.
 $$X\cong U_2 \xrightarrow{f_1} X_1 \cong Q_5 \xrightarrow{f_2} X_2 \cong L_3 \xrightarrow{g} Y\cong \mathbb{F}_1 \times \mathbb{P}^1.$$
 
 \item Let $X_2:=L_{11} \cong \mathbb{P}^1 \times \mathbb{P}_{\mathbb{P}^1\times \mathbb{P}^1}(\mathcal{O}(0,-1)\oplus \mathcal{O}(-1,0))$. The primitive generators of the rays of the fan of $X_2$ are given by\footnote{Let us remark that there is a misprint in \cite[\S 4]{BAT99}: the variety $L_{11}$ is isomorphic to $\mathbb{P}^1\times V(\mathcal{C}_5)$ in Batyrev's notation, i.e. $L_{11} \cong \mathbb{P}^1 \times \mathbb{P}_{\mathbb{P}^1\times \mathbb{P}^1}(\mathcal{O}(0,-1)\oplus \mathcal{O}(-1,0))$.}
 \begin{equation*}\begin{split}\Delta_{X_2}(1)=\{ & u_1 = e_1, u_2 = e_2, u_3 = -e_2, u_4=e_3,\\ & u_5=-e_2 - e_3,u_6 = e_4,u_7=e_2 - e_4,u_8=-e_1 \}. \end{split}\end{equation*}
 The elementary conic bundle $g:L_{11} \to \mathbb{P}^1 \times \mathbb{P}_{\mathbb{P}^1\times \mathbb{P}^1}(\mathcal{O}(0,-1)\oplus \mathcal{O}(-1,0)) $ is induced by the quotient map $G:\mathbb{R}^4 \to \mathbb{R}^4\slash \langle e_1 \rangle$. Let $A_1 = V(u_1,u_3), A_2=V(u_2,u_8)$, where $A_i \cong \mathbb{P}^1 \times \mathbb{P}^1$ for $i=1,2$. We verify with Macaulay2 that $X_1:=\operatorname{Bl}_{A_1}(X_2)\cong Q_{16}$ and $X:=\operatorname{Bl}_{A_1,A_2}(X_2)\cong U_8$ are Fano. We get the following factorization of $f\colon X\to Y$.
 $$X\cong U_8 \xrightarrow{f_1} X_1 \cong Q_{16} \xrightarrow{f_2} X_2 \cong L_{11} \xrightarrow{g} Y \cong \mathbb{P}_{\mathbb{P}^1\times \mathbb{P}^1}(\mathcal{O}(0,-1)\oplus \mathcal{O}(-1,0)).$$
\end{enumerate}
\end{exem}

\begin{rema}\label{Remark: Toric product}
As pointed out in Remark \ref{case_product}, the case of products of two del Pezzo surfaces is well-understood. In the toric case, there are examples given by $f\colon S_{1}\times \mathbb{P}^1\times \mathbb{P}^1\to \mathbb{P}^1\times \mathbb{P}^1\times \mathbb{P}^1$, and by $f\colon S_{1}\times \mathbb{P}^2\to \mathbb{P}^1\times \mathbb{P}^2$ both induced by a conic bundle $S_{1}\to \mathbb{P}^1$ with $\rho_{S_{1}}=4$.
\end{rema}

\begin{rema}
Notice that in all known examples of Fano conic bundle $f\colon X\to Y$ which satisfy assumptions of Theorem \ref{main_thm}, $X$ is a toric variety. Taking into account Corollary \ref{cor_1}, having some examples with $X$ non toric could be interesting from the viewpoint of rationality problem for conic bundles. We refer the reader to the recent article \cite{PRO17} for a survey on this problem.
\end{rema}

\section*{Acknowledgements}

This collaboration started during the conference "New Advances in Fano {Manifolds}" held in the University of Cambridge. Part of this work was developed during the first author's visit at MIMUW, University of Warsaw. 

We would like to thank both institutions for their support and hospitality, and to the organizers of the conference for making this event successful. We are very grateful to Cinzia \textsc{Casagrande}, Grzegorz \textsc{Kapustka} and Jaros{\l}aw \textsc{Wi\'sniewski} for many valuable conversations. 

We also express our thanks to the anonymous referees for their careful reading of our paper, and for the important suggestions which have allowed us to improve it. 

The first author was founded by Hua Loo-Keng Center for Mathematical Sciences, AMSS, CAS (National Natural Science Foundation of China No. 11688101). The second author was supported by Polish National Science Center grants 2013/08/A/ST1/00804 and 2016/23/G/ST1/04828.


\end{document}